\documentclass[a4paper,11pt]{article}
\usepackage{amsmath,amssymb,amsthm,easywncy}
\usepackage[all]{xy}

\theoremstyle{plain}
\newtheorem{thm}{Theorem}[section]
\newtheorem{lem}[thm]{Lemma}
\newtheorem{prop}[thm]{Proposition}
\newtheorem{cor}[thm]{Corollary}
\newtheorem{conj}[thm]{Conjecture}

\theoremstyle{definition}
\newtheorem{defn}[thm]{Definition}

\newtheorem{rem}[thm]{Remark}
\newtheorem{ex}[thm]{Example}

\numberwithin{equation}{section}

\newcommand{\Q}{\mathbb{Q}}
\newcommand{\R}{\mathbb{R}}
\newcommand{\bk}{\mathbf{k}}
\newcommand{\bl}{\mathbf{l}}
\newcommand{\frH}{\mathfrak{H}}
\newcommand{\frP}{\mathfrak{P}}

\newcommand{\sh}{\mathcyr{sh}}
\newcommand{\bsh}{\mathbin{\mathcyr{sh}}}
\newcommand{\bast}{\mathbin{\bar{*}}}
\newcommand{\Zbar}{Z^\star}
\newcommand{\rhobar}{\rho^\star}
\newcommand{\rev}[1]{\overleftarrow{#1}}
\newcommand{\dual}{\dagger}
\newcommand{\trp}{\mathrm{t}}

\begin{document}
\title{A new integral-series identity of multiple zeta values \\
and regularizations}
\author{Masanobu Kaneko and Shuji Yamamoto}
\date{\today}
\maketitle 

\begin{abstract}
We present a new ``integral $=$ series'' type identity of multiple zeta values, 
and show that this is equivalent in a suitable sense to the fundamental theorem of regularization. 
We conjecture that this identity is enough to describe 
{\it all} linear relations of multiple zeta values over $\Q$. 
We also establish the regularization theorem for multiple zeta-star values, 
which too is equivalent to our new identity.  
A connection to Kawashima's relation is discussed as well. 
\end{abstract}

\section{Introduction}
The multiple zeta values (MZVs) and the multiple zeta-star values 
(MZSVs, or sometimes referred to as the non-strict MZVs) are 
defined respectively by the nested series
\begin{align*}
\zeta(k_1,\ldots, k_r)&=\sum_{0<m_1<\cdots<m_r}
\frac{1}{m_1^{k_1}\cdots m_r^{k_r}} \\
\intertext{and}
\zeta^\star(k_1,\ldots, k_r)&=\sum_{0<m_1\leq \cdots\leq m_r}
\frac{1}{m_1^{k_1}\cdots m_r^{k_r}},  
\end{align*}
where $k_i\ (1\le i\le r)$ are arbitrary positive integers with $k_r\ge2$ 
(to ensure the convergence). 

These numbers appear in various branches of mathematics as well as mathematical physics, and 
have been actively studied since more than two decades. One of the main points of interest in those studies 
is to find as many relations of MZVs as possible, and to pin down concretely the set of 
basic relations which describe all linear or algebraic relations of MZVs over $\Q$. 
Several candidates of such a set are known today, 
of which we only mention here the ``associator relations'' 
and the ``extended double shuffle relations.'' 
However, whether any of them give {\it all} relations 
is still conjectural and unknown so far.  

In this paper, we present a new, very simple and elementary relation 
which conjecturally supplies all linear relations. 
The form of the relation (Theorem \ref{prop:SI}) is 
\[\zeta\bigl(\mu(\bk,\bl)\bigr)=
\sum_{0<m_1<\cdots<m_r=n_s\geq\cdots\geq n_1>0}
\frac{1}{m_1^{k_1}\cdots m_r^{k_r}n_1^{l_1}\cdots n_s^{l_s}}, \]
where $\bk=(k_1,\ldots,k_r)$ and $\bl=(l_1,\ldots,l_s)$ are 
any arrays of positive integers, and the left-hand side is 
a certain integral which can be written, likewise the sum on the right, 
as a linear combination of MZVs.  
See \S\S \ref{sec:poset}, \ref{sec:main} for precise definition. 

We also show that these relations, together with either the shuffle or the harmonic (or stuffle) product formula, are equivalent to the 
extended (or regularized) double shuffle relation (Theorem \ref{thm:equiv} and Theorem \ref{thm:shuffle=harmonic}), 
which is conjectured to give all relations of MZVs. 
A way of regularizing the MZSVs is also newly introduced, and 
Theorem \ref{thm:equiv} contains the regularization theorem of MZSVs as well.

We first give necessary preliminaries in \S\ref{sec:notation} and \S\ref{sec:poset}, and then state our 
main theorems in \S\ref{sec:main}. The proof of Theorem \ref{prop:SI} is given immediately after stating the 
theorem, whereas the proof of Theorem \ref{thm:equiv} is separately given in \S\ref{sec:proof}.  
In \S\ref{sec:kawashima}, we discuss a relation between our theorems and Kawashima's relation. 
Assuming the duality, Kawashima's relation is also deduced from our main identity. 
In the final \S\ref{sec:sumformula}, we deduce the restricted sum formula 
of Eie-Liaw-Ong \cite{ELO} from our identity. 

\section{Notation and algebraic setup}\label{sec:notation}
A finite sequence $\bk=(k_1,\ldots,k_r)$ of positive integers is called 
an \emph{index}. Here the length $r$ (called the \emph{depth} of $\bk$) 
can be $0$ and the unique index of depth $0$, namely the empty sequence,
is denoted by $\varnothing$.  An index $\bk=(k_1,\ldots,k_r)$ is \emph{admissible} if $k_r\geq 2$. The index 
$\varnothing$ is also regarded as an admissible index. The sum of components of an index  is called \emph{weight},
and $\varnothing$ is thought to be of weight 0. 
 
Then, as already defined in the introduction, 
multiple zeta and zeta-star values 
associated to an admissible index $\bk=(k_1,\ldots,k_r)$ are given by 
\begin{align*}
\zeta(\bk)&=\sum_{0<m_1<\cdots<m_r}
\frac{1}{m_1^{k_1}\cdots m_r^{k_r}} \\
\intertext{and}
\zeta^\star(\bk)&=\sum_{0<m_1\leq \cdots\leq m_r}
\frac{1}{m_1^{k_1}\cdots m_r^{k_r}} 
\end{align*}
respectively. We set $\zeta(\varnothing)=\zeta^\star(\varnothing)=1$. 

We recall Hoffman's algebraic setup \cite{H} with a slightly different 
convention. Write $\frH=\Q\langle e_0,e_1\rangle$ for 
the noncommutative polynomial algebra 
of indeterminates $e_0$ and $e_1$ over $\Q$, 
and define its subalgebras $\frH^0$ and $\frH^1$ by 
\[\frH^0=\Q+e_1\frH e_0\subset \frH^1=\Q+e_1\frH\subset \frH. \]
We put $e_k=e_1e_0^{k-1}$ for any positive integer $k$, so that 
the monomials $e_{k_1}\cdots e_{k_r}$ associated to all indices 
(resp.\ admissible indices) $(k_1,\ldots,k_r)$ form a basis of $\frH^1$ 
(resp.\ $\frH^0$) over $\Q$ (the monomial associated to $\varnothing$ is 1). We often identify an index $(k_1,\ldots,k_r)$ 
with the monomial $e_{k_1}\cdots e_{k_r}$ in $\frH^1$. 

We consider two $\Q$-bilinear commutative products $\bsh$ 
on $\frH$ and $*$ on $\frH^1$, 
called the shuffle and the harmonic (or stuffle) products, 
which are characterized by the bilinearity and the recursive formulas
\begin{gather*}
1\bsh w=w\bsh 1=w\quad (w\in\frH), \\
av\bsh bw=a(v\bsh bw)+b(av\bsh w)\quad (a,b\in\{e_0,e_1\}, v,w\in\frH), 
\end{gather*}
and 
\begin{gather*}
1*w=w*1=w\quad (w\in\frH^1), \\
e_kv*e_lw=e_k(v*e_lw)+e_l(e_kv*w)+e_{k+l}(v*w)\quad (k,l\geq 1, v,w\in\frH^1), 
\end{gather*}
respectively. We denote by $\frH_\sh$ (resp.\ $\frH_*^1$) 
the commutative $\Q$-algebra $\frH$ (resp.\ $\frH^1$) equipped with multiplication $\sh$ (resp.\ $*$). 
Then the subspaces $\frH^1$ and $\frH^0$ of $\frH$ (resp. the subspace $\frH^0$ of $\frH^1$) are closed under
$\sh$ (resp. $*$) and become subalgebras of $\frH_\sh$ (resp. $\frH_*^1$)  
denoted by $\frH_\sh^1$ and $\frH_\sh^0$ (resp.  $\frH_*^0$). 

If with each index $\bk$ is assigned an element $f(\bk)$ of 
a $\Q$-vector space $V$, we often extend this assignment to a $\Q$-linear map 
from $\frH^1$ (or $\frH^0$ or $e_1\frH$, depending on 
the range of definition of $f$) to $V$ and denote this extension 
by the same symbol $f$. The typical example is the $\Q$-linear map 
$\zeta\colon\frH^0\to\R$ extending the definition of multiple zeta values. 
In the sequel we freely use the letter $\zeta$ with this extended meaning. 

The map $\zeta\colon\frH^0\to\R$ is a $\Q$-algebra homomorphism 
on both $\frH^0_\sh$ and $\frH^0_*$, that is, in terms of indices, 
\begin{equation}\label{doubleshuffle}
\zeta(\bk\bsh\bl)=\zeta(\bk*\bl)=\zeta(\bk)\zeta(\bl)
\end{equation}
for any admissible indices $\bk$ and $\bl$. 
This is the \emph{double shuffle relation} of MZVs. 

Next we briefly review the theory of regularization of multiple zeta values. 
Because of the isomorphisms $\frH_\sh^1\cong \frH^0_\sh[e_1]$ and 
$\frH^1_*\cong\frH^0_*[e_1]$ 
(see \cite[Theorem 6.1]{R} and \cite[Theorem 2.6]{H}), 
we can extend the map $\zeta$ uniquely to $\Q$-algebra homomorphisms 
$\zeta_\sh\colon\frH_\sh^1\to\R[T]$ and $\zeta_*\colon\frH^1_*\to\R[T]$ 
from $\frH^1$ to the polynomial algebra $\R[T]$ by setting $\zeta_\sh(e_1)=\zeta_*(e_1)=T$. 
The extended maps $\zeta_\sh$ and $\zeta_*$ are called the \emph{shuffle} and 
\emph{harmonic regularizations} of $\zeta$, respectively. 
When we need to make the indeterminate $T$ explicit, 
we write $\zeta_\bullet(w;T)$ for $\zeta_\bullet(w)\in\R[T]$, 
where $\bullet=\sh$ or $*$, and we also write $\zeta_\bullet(w;T)$ 
as $\zeta_\bullet(\bk;T)$ when the index $\bk$ corresponds to the word $w$. 

The fundamental theorem of regularizations of MZVs then asserts that the two polynomials 
$\zeta_\sh(\bk;T)$ and $\zeta_*(\bk;T)$ are related with each other by a simple $\R$-linear map 
coming from the Taylor series of the gamma function $\Gamma(u)$. 
Define an $\R$-linear endomorphism $\rho$ on $\R[T]$ by the equality 
\begin{equation}\label{rho}
\rho(e^{Tu})=A(u)e^{Tu} 
\end{equation}
in the formal power series algebra $\R[T][[u]]$ on which $\rho$ acts 
coefficientwise, where 
\[A(u)=\exp\Biggl(\sum_{n=2}^\infty
\frac{(-1)^n}{n}\zeta(n)u^n\Biggr)\in\R[[u]]. \]
Note that $A(u)=e^{\gamma u}\Gamma(1+u)$, where $\gamma$ is Euler's constant. 

\begin{thm}[{\cite[Theorem 1]{IKZ}}] \label{thm:IKZ}
For any index $\bk$, we have 
\begin{equation}\label{regfund}
\zeta_\sh(\bk;T)=\rho\bigl(\zeta_*(\bk;T)\bigr).  
\end{equation}
\end{thm}

It is conjectured that this relation 
(or more precisely the relations obtained by comparing the coefficients), 
together with the double shuffle relation \eqref{doubleshuffle}, 
describes all (algebraic and linear) relations of MZVs over $\Q$.  

For a non-empty index $\bk=(k_1,\ldots,k_r)$, we write $\bk^\star$ 
for the formal sum of $2^{r-1}$ indices of the form 
$(k_1\bigcirc \cdots \bigcirc k_r)$, where each $\bigcirc$ is 
replaced by `\,,\,' or `+'. 
We also put $\varnothing^\star=\varnothing$. 
Then $\bk^\star$ is identified with an element of $\frH^1$, 
and we have $\zeta^\star(\bk)=\zeta(\bk^\star)$ for admissible $\bk$. 

Finally, we introduce the $\Q$-bilinear `circled harmonic product' 
$\circledast\colon e_1\frH\times e_1\frH \to e_1\frH e_0$
defined by 
\[ve_k\circledast we_l=(v*w)e_{k+l}\quad (k,l\geq 1, v,w\in\frH^1). \]
This is a binary operation on the space of formal sums of non-empty indices 
taking values in the subspace spanned by non-empty admissible indices. 
We readily see from the definition that, 
for non-empty indices $\bk=(k_1,\ldots,k_r)$ and 
$\bl=(l_1,\ldots,l_s)$, we have the series expression
\begin{equation}\label{eq:cast}
\zeta\bigl(\bk\circledast\bl^\star\bigr)=
\sum_{0<m_1<\cdots<m_r=n_s\geq\cdots\geq n_1>0}
\frac{1}{m_1^{k_1}\cdots m_r^{k_r}n_1^{l_1}\cdots n_s^{l_s}}. 
\end{equation}
To see this, we put $n=m_r (=n_s)$ and write the right-hand side as
\[ \sum_{n=1}^\infty\biggl( \sum_{0<m_1<\cdots<m_{r-1}<n}\frac{1}{m_1^{k_1}\cdots m_{r-1}^{k_{r-1}}}
\biggr)\biggl(\sum_{0<n_1\le\cdots \le n_{s-1}\le n} \frac{1}{n_1^{l_1}\cdots n_{s-1}^{l_{s-1}}}\biggr)\frac1{n^{k_r+l_s}}, \]
and note that the product of truncated sums for a fixed $n$ obey the harmonic product rule.

The formula \eqref{eq:cast} includes MZV and MZSV as special cases: 
\begin{align*}
\zeta\bigl(\bk\circledast (1)^\star)&=\zeta(k_1,\ldots,k_{r-1},k_r+1)\\
\intertext{and}
\zeta\bigl((1)\circledast \bl^\star)&=\zeta^\star(l_1,\ldots,l_{s-1},l_s+1). 
\end{align*}

\section{Review on 2-posets and associated integrals}\label{sec:poset}

In this section, we review the definitions and basic properties of 2-labeled posets 
(in this paper, we call them 2-posets for short) 
and the associated integrals introduced by the second-named author in \cite{Y}. 
\begin{defn}
A \textit{2-poset} is a pair $(X,\delta_X)$, where $X=(X,\leq)$ is 
a finite partially ordered set (poset for short) and 
$\delta_X$ is a map from $X$ to $\{0,1\}$. 
We often omit  $\delta_X$ and simply say ``a 2-poset $X$.'' 
The $\delta_X$ is called the \emph{label map} of $X$. 

A 2-poset $(X,\delta_X)$ is called \textit{admissible} if 
$\delta_X(x)=0$ for all maximal elements $x\in X$ and 
$\delta_X(x)=1$ for all minimal elements $x\in X$. 
\end{defn}

A 2-poset is depicted as a Hasse diagram in which an element $x$ with 
$\delta(x)=0$ (resp. $\delta(x)=1$) is represented by $\circ$ 
(resp. $\bullet$). For example, the diagram 
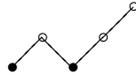
\[\begin{xy}
{(0,-4) \ar @{{*}-o} (4,0)},
{(4,0) \ar @{-{*}} (8,-4)},
{(8,-4) \ar @{-o} (12,0)},
{(12,0) \ar @{-o} (16,4)}
\end{xy} \]
represents the 2-poset $X=\{x_1,x_2,x_3,x_4,x_5\}$ with order 
$x_1<x_2>x_3<x_4<x_5$ and label 
$(\delta_X(x_1),\ldots,\delta_X(x_5))=(1,0,1,0,0)$. 
This 2-poset is admissible. 

\begin{defn}
For an admissible 2-poset $X$, we define the associated integral 
\begin{equation}\label{eq:I(X)}
I(X)=\int_{\Delta_X}\prod_{x\in X}\omega_{\delta_X(x)}(t_x), 
\end{equation}
where 
\[\Delta_X=\bigl\{(t_x)_x\in [0,1]^X \bigm| t_x<t_y \text{ if } x<y\bigr\}\]
and 
\[\omega_0(t)=\frac{dt}{t}, \quad \omega_1(t)=\frac{dt}{1-t}. \]
\end{defn}
Note that the admissibility of a 2-poset corresponds to 
the convergence of the associated integral. 

\begin{ex}\label{ex:MZV}
When an admissible 2-poset is totally ordered, 
the corresponding integral is exactly the iterated integral expression 
for a multiple zeta value. 
To be precise, for an index $\bk=(k_1,\ldots,k_r)$ (admissible or not), 
we write
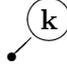
\[\begin{xy}
{(5,2)*++[o][F]{\bk}="k"}, 
{(0,-3) \ar @{{*}-} "k"}, 
\end{xy}\]
for the `totally ordered' diagram: 
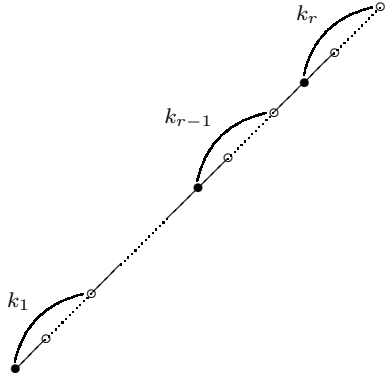
\[\begin{xy}
{(0,-24) \ar @{{*}-o} (4,-20)}, 
{(4,-20) \ar @{.o} (10,-14)}, 
{(10,-14) \ar @{-} (14,-10)}, 
{(14,-10) \ar @{.} (20,-4)}, 
{(20,-4) \ar @{-{*}} (24,0)}, 
{(24,0) \ar @{-o} (28,4)}, 
{(28,4) \ar @{.o} (34,10)}, 
{(34,10) \ar @{-{*}} (38,14)}, 
{(38,14) \ar @{-o} (42,18)}, 
{(42,18) \ar @{.o} (48,24)}, 
{(0,-23) \ar @/^2mm/ @{-}^{k_1} (9,-14)}, 
{(24,1) \ar @/^2mm/ @{-}^{k_{r-1}} (33,10)}, 
{(38,15) \ar @/^2mm/ @{-}^{k_r} (47,24)} 
\end{xy} \]
If $k_i=1$, we understand the notation $\begin{xy}
{(0,-5) \ar @{{*}-o} (4,-1)}, 
{(4,-1) \ar @{.o} (10,5)}, 
{(0,-4) \ar @/^2mm/ @{-}^{k_i} (9,5)} 
\end{xy}$ as a single $\bullet$, and if $\bk=\varnothing$, we regard the diagram as the empty 2-poset. 

Then, if $\bk$ is an admissible index, we have 
\begin{equation}\label{eq:SI MZV}
\zeta(\bk)=
I\left(\begin{xy}
{(5,2)*++[o][F]{\bk}="k"}, 
{(0,-3) \ar @{{*}-} "k"}, 
\end{xy}\right),
\end{equation}
and any totally ordered admissible 2-poset is of the form $\begin{xy}
{(5,2)*++[o][F]{\bk}="k"}, 
{(0,-3) \ar @{{*}-} "k"}, 
\end{xy}$ with an admissible index $\bk$.
\end{ex}

\begin{ex}\label{ex:MZSV}
In \cite{Y}, an integral expression for multiple zeta-star values 
is described in terms of a 2-poset. 
For an index $\bl=(l_1,\ldots,l_s)$, we write 
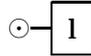
\[\begin{xy}
{(7,0)*++[F]{\bl}="l"}, 
{(0,0)*{\odot} \ar @{-} "l"}, 
\end{xy}\]
for the following diagram: 
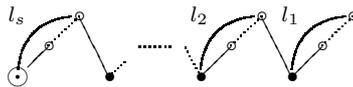
\[\begin{xy}
{(0,-4) *{\odot} \ar @{-o} (4,0)}, 
{(4,0) \ar @{.o} (8,4)}, 
{(8,4) \ar @{-{*}} (12,-4)}, 
{(12,-4) \ar @{.} (14,-2)}, 
{(16,0) \ar @{.} (20,0)}, 
{(22,0) \ar @{.{*}} (24,-4)}, 
{(24,-4) \ar @{-{o}} (28,0)}, 
{(28,0) \ar @{.o} (32,4)}, 
{(32,4) \ar @{-{*}} (36,-4)}, 
{(36,-4) \ar @{-{o}} (40,0)}, 
{(40,0) \ar @{.o} (44,4)}, 
{(0,-3) \ar @/^2mm/ @{-}^{l_s} (7,4)}, 
{(24,-3) \ar @/^2mm/ @{-}^{l_2} (31,4)}, 
{(36,-3) \ar @/^2mm/ @{-}^{l_1} (43,4)}, 
\end{xy}\]
where the symbol $\odot$ represents either $\circ$ or $\bullet$. 
For example, 
\[\begin{xy}
{(10,0)*+[F]{(2,3)}="l"}, 
{(0,0) \ar @{{*}-} "l"}, 
\end{xy}
\ =\ 
\begin{xy}
{(0,-4) \ar @{{*}-} (4,0)}, 
{(4,0) \ar @{o-} (8,4)}, 
{(8,4) \ar @{o-} (12,0)}, 
{(12,0) \ar @{{*}-o} (16,4)}, 
\end{xy}\ ,
\qquad 
\begin{xy}
{(12,0)*+[F]{(3,1,1)}="l"}, 
{(0,0) \ar @{o-} "l"}, 
\end{xy}
\ =\ 
\begin{xy}
{(0,4) \ar @{o-} (4,0)}, 
{(4,0) \ar @{{*}-} (8,-4)}, 
{(8,-4) \ar @{{*}-} (12,0)}, 
{(12,0) \ar @{o-o} (16,4)}, 
\end{xy} \ .\]
Note our convention in the case when some  $l_i$ is 1.  The line from the $\bullet$
corresponding to $l_i=1$ goes down. 

Then, if $\bl$ is an admissible index, we have 
\begin{equation}\label{eq:SI MZSV}
\zeta^\star(\bl)=
I\left(\ \begin{xy}
{(7,0)*++[F]{\bl}="l"}, 
{(0,0) \ar @{{*}-} "l"}, 
\end{xy}\ \right). 
\end{equation}
\end{ex}

We also recall an algebraic setup for 2-posets 
(cf.\ Remark at the end of \S2 of \cite{Y}). 
Let $\frP$ be the $\Q$-algebra generated by 
the isomorphism classes of 2-posets, 
whose multiplication is given by the disjoint union of 2-posets. 
Then the integral \eqref{eq:I(X)} defines a $\Q$-algebra homomorphism 
$I\colon\frP^0\to\R$ from the subalgebra $\frP^0$ of $\frP$ generated by 
the classes of admissible 2-posets. 

Moreover, there is a unique $\Q$-algebra homomorphism $W\colon\frP\to\frH_\sh$ 
which satisfies the following two conditions: 
\begin{enumerate}
\item[1)] If (the underlying poset of) a 2-poset $X=\{x_1<x_2<\cdots<x_k\}$ is totally ordered, 
\[W(X)=e_{\delta_X(x_1)}e_{\delta_X(x_2)}\cdots e_{\delta_X(x_k)}.\]
\item[2)] If $a$ and $b$ are non-comparable elements of a 2-poset $X$, 
the identity 
\begin{equation}\label{eq:shuffle for W}
W(X)=W(X^b_a)+W(X^a_b)
\end{equation}
holds. Here $X^b_a$ denotes the 2-poset that is obtained from $X$ 
by adjoining the relation $a<b$ (see \cite[Definition 2.2 (2)]{Y}).  
\end{enumerate}
Then we have $W(\frP^0)=\frH^0$ and $I=\zeta\circ W\colon\frP^0\to\R$. 

\begin{ex}
\[W\left(\ 
\begin{xy}
{(0,-3) \ar@{{*}-} (3,0)}, 
{(3,0) \ar@{{*}-o} (6,3)}
\end{xy}\ \right)
=e_1e_1e_0, \quad 
W\left(\ 
\begin{xy}
{(0,0) \ar@{{*}-} (3,3)}, 
{(3,3) \ar@{o-} (6,0)}, 
{(0,0) \ar@{-} (3,-3)}, 
{(3,-3) \ar@{{*}-o} (6,0)}, 
\end{xy}\ \right)
=e_1e_1e_0e_0+e_1e_0e_1e_0. \]
\end{ex}

\section{Main theorems}\label{sec:main}

For non-empty indices $\bk=(k_1,\ldots,k_r)$ and $\bl=(l_1,\ldots,l_s)$, 
we put 
\begin{equation}\label{eq:mu}
\mu(\bk,\bl)
=W\left(\ \begin{xy}
{(6,0)*++[o][F]{\bk}="k"},
{(20,6)*++[F]{\bl}="l"},
{(0,-6) \ar @{{*}-} "k"}, 
{"k" \ar @{-{o}} (12,6)}, 
{(12,6) \ar @{-} "l"}, 
\end{xy}
\ \right)\in\frH^0. 
\end{equation}
Here, the diagram on the right-hand side is a combination of 
the symbols introduced in Examples \ref{ex:MZV} and \ref{ex:MZSV} 
and represents the following: 
\[\begin{xy}
{(-3,-18) \ar @{{*}-} (0,-15)}, 
{(0,-15) \ar @{{o}.} (3,-12)}, 
{(3,-12) \ar @{{o}.} (9,-6)}, 
{(9,-6) \ar @{{*}-} (12,-3)}, 
{(12,-3) \ar @{{o}.} (15,0)}, 
{(15,0) \ar @{{o}-} (18,3)}, 
{(18,3) \ar @{{o}-} (21,6)}, 
{(21,6) \ar @{{o}.} (24,9)}, 
{(24,9) \ar @{{o}-} (27,3)}, 
{(27,3) \ar @{{*}-} (30,6)}, 
{(30,6) \ar @{{o}.} (33,9)}, 
{(33,9) \ar @{{o}-} (35,5)}, 
{(37,6) \ar @{.} (41,6)}, 
{(42,3) \ar @{{*}-} (45,6)}, 
{(45,6) \ar @{{o}.{o}} (48,9)}, 
{(-3,-17) \ar @/^1mm/ @{-}^{k_1} (2,-12)}, 
{(9,-5) \ar @/^1mm/ @{-}^{k_r} (14,0)}, 
{(18,4) \ar @/^1mm/ @{-}^{l_s} (23,9)}, 
{(28,3) \ar @/_1mm/ @{-}_{l_{s-1}} (33,8)}, 
{(43,3) \ar @/_1mm/ @{-}_{l_1} (48,8)}, 
\end{xy}\]
When $s=1$, we understand this as
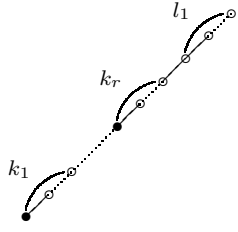
\[\begin{xy}
{(-3,-18) \ar @{{*}-} (0,-15)}, 
{(0,-15) \ar @{{o}.} (3,-12)}, 
{(3,-12) \ar @{{o}.} (9,-6)}, 
{(9,-6) \ar @{{*}-} (12,-3)}, 
{(12,-3) \ar @{{o}.} (15,0)}, 
{(15,0) \ar @{{o}-} (18,3)}, 
{(18,3) \ar @{{o}-} (21,6)}, 
{(21,6) \ar @{{o}.} (24,9)}, 
{(24,9) \ar @{.{o}} (24,9)}, 
{(-3,-17) \ar @/^1mm/ @{-}^{k_1} (2,-12)}, 
{(9,-5) \ar @/^1mm/ @{-}^{k_r} (14,0)}, 
{(18,4) \ar @/^1mm/ @{-}^{l_1} (23,9)}
\end{xy}\]

Notice that the leftmost vertex of the `$\bl$-part', which is located between 
$\begin{xy}*+[Fo]{\bk}\end{xy}$ and $\begin{xy}*+[F]{\bl}\end{xy}$ 
in \eqref{eq:mu}, is $\circ$. 

\bigskip

Our first main theorem is the following identity which generalizes both \eqref{eq:SI MZV} 
and \eqref{eq:SI MZSV}. 

\begin{thm}\label{prop:SI}
For any non-empty indices $\bk$ and $\bl$, we have 
\begin{equation}\label{eq:SI}
\zeta\bigl(\mu(\bk,\bl)\bigr)=\zeta\bigl( \bk \circledast \bl^\star \bigr). 
\end{equation}
\end{thm}
\begin{proof}
The proof is done straightforwardly by computing the multiple integral 
as a repeated integral ``from left to right.'' 
More specifically, one first computes the integral of the $\bk$-part 
as in the usual proof of \eqref{eq:SI MZV}, 
and then computes the integral of the $\bl$-part as in 
the proof of \cite[Theorem 1.2]{Y}.  We do not write out the details here, but the machinery
should be clear enough in the computation of the following example.
\end{proof}

\begin{ex}  
For $\bk=(1,1)$ and $\bl=(2,1)$, we have 
\[\mu(\bk,\bl)
=W\left(\ \begin{xy}
{(0,-4) \ar @{{*}-} (4,0)}, 
{(4,0) \ar @{{*}-} (8,4)}, 
{(8,4) \ar @{o-} (12,0)}, 
{(12,0) \ar @{{*}-o} (16,4)}, 
\end{xy}
\ \right)
=6e_1e_1e_3+2e_1e_2e_2+e_2e_1e_2\]
and 
\[\bk\circledast\bl^\star=e_1e_1\circledast e_2e_1+
e_1e_1\circledast e_3
=e_1e_2e_2+e_2e_1e_2+e_3e_2+e_1e_4. \]
Hence the identity \eqref{eq:SI} gives a linear relation
\[6\zeta(1,1,3)+2\zeta(1,2,2)+\zeta(2,1,2)
=\zeta(1,2,2)+\zeta(2,1,2)+\zeta(3,2)+\zeta(1,4). \]
The proof of \eqref{eq:SI} in this case proceeds as follows 
(here we omit the condition $0<t_i<1$ from the notation): 
\begin{align*}
\zeta\bigl(\mu(\bk,\bl)\bigr)
&=\int_{t_1<t_2<t_3>t_4<t_5}
\frac{dt_1}{1-t_1}\frac{dt_2}{1-t_2}\frac{dt_3}{t_3}\frac{dt_4}{1-t_4}\frac{dt_5}{t_5} 
=\int_{t_2<t_3>t_4<t_5}
\sum_{l=1}^\infty\frac{t_2^{l}}{l}\frac{dt_2}{1-t_2}\frac{dt_3}{t_3}
\frac{dt_4}{1-t_4}\frac{dt_5}{t_5}\\
&=\int_{t_3>t_4<t_5}
\sum_{l,m=1}^\infty\frac{t_3^{l+m}}{l(l+m)}\frac{dt_3}{t_3}\frac{dt_4}{1-t_4}\frac{dt_5}{t_5}
=\int_{t_4<t_5}
\sum_{l,m=1}^\infty\frac{1-t_4^{l+m}}{l(l+m)^2}\frac{dt_4}{1-t_4}\frac{dt_5}{t_5}\\
&=\int_{t_4<t_5}
\sum_{0<m_1<m_2}\frac{1}{m_1m_2^2}\sum_{n=1}^{m_2}t_4^{n-1}dt_4\frac{dt_5}{t_5}
=\int_0^1
\sum_{0<m_1<m_2}\frac{1}{m_1m_2^2}\sum_{n=1}^{m_2}\frac{t_5^n}{n}\frac{dt_5}{t_5}\\
&=\sum_{0<m_1<m_2\geq n>0}\frac{1}{m_1m_2^2n^2}
=\sum_{0<m_1<m_2=n_2\geq n_1>0}\frac{1}{m_1m_2n_1^2n_2}
=\zeta\bigl(\bk\circledast\bl^\star\bigr). 
\end{align*}
At each $\bullet$ in the $\bl$-part of the repeated integral, an integral of the form $\int_0^{t''}\int_{t'}^1\frac{t^n}{n^{k-1}}\frac{dt}{t}
\frac{dt'}{1-t'}$ appears and this gives the ``$\star$-part" sum $\sum_{1\le n'\le n}\frac1{n'n^k}$.

\end{ex}

The integral on the left-hand side of \eqref{eq:SI} is a sum of MZVs (sum over
all integrals associated to possible total-order extension of the 2-poset in \eqref{eq:mu},
see \cite[Corollary 2.4]{Y}), whereas the right-hand side is also a sum of MZVs
in the usual way.  Hence, for any given (non-empty) indices $\bk$ and $\bl$, 
the identity gives a linear relation among MZVs.  We remark that if the depth $s$ of $\bl$ is greater than 1, this linear relation 
is {\it always non-trivial} because whereas the depth of MZVs appear on the left are always the 
sum of depths of $\bk$ and $\bl$,  there must appear terms of lesser depth on the right when $s>1$.   
We conjecture that the totality of these relations \eqref{eq:SI} gives {\it all} linear relations among MZVs:

\begin{conj}\label{ISconj}
Any linear dependency of MZVs over $\Q$ can be deduced from \eqref{eq:SI}
with some $\bk$s and $\bl$s.
\end{conj}

We checked by computer up to weight $17$ that \eqref{eq:SI} is enough to reduce 
the dimension of the space of MZVs of fixed weight to the well-known conjectural dimension.  
In addition to the numerical evidence, in view of the widely believed conjecture that the double shuffle relation \eqref{doubleshuffle} 
and the regularization theorem \eqref{regfund} describe all algebraic and linear relations of MZVs, 
our theorems below (Theorem \ref{thm:equiv} and Theorem \ref{thm:shuffle=harmonic})
give a strong support to our conjecture.
We would like to stress that equation \eqref{eq:SI} is a completely elementary identity 
between convergent integral and sum, without any process of regularization, 
even for non-admissible $\bk$ or $\bl$. 

Our second main theorem claims that the identity \eqref{eq:SI} is, 
in a suitable sense, equivalent to the fundamental theorem of regularization 
\eqref{regfund}.  We also formulate the zeta-star version of \eqref{regfund} 
and show that this too is equivalent. To state the theorem, 
we first introduce a formal setting. 

\begin{defn}
Let $Z\colon\frH^0\to\R$ be a $\Q$-linear map. 
We say that $Z$ satisfies the \emph{double shuffle relation} 
if it is both a $\Q$-algebra homomorphism 
from $\frH_\sh^0$ and from $\frH_*^0$, i.e., the relation 
\[Z(v\bsh w)=Z(v*w)=Z(v)Z(w)\]
holds for any $v,w\in\frH^0$. 
\end{defn}
This is of course modeled after our MZV-evaluation map 
$\zeta\colon\frH^0\to\R$ satisfying \eqref{doubleshuffle}. 

Suppose that $Z$ satisfies the double shuffle relation. 
Then, just as in the case of $\zeta$, 
we can extend the map $Z$ uniquely to $\Q$-algebra homomorphisms 
$Z_\sh\colon\frH_\sh^1\to\R[T]$ and $Z_*\colon\frH^1_*\to\R[T]$ 
by setting $Z_\sh(e_1)=Z_*(e_1)=T$. 
The maps $Z_\sh$ and $Z_*$ are called the \emph{shuffle} and 
\emph{harmonic regularization} of $Z$, respectively. 
When we need to make the variable $T$ explicit, 
we write $Z_\bullet(w;T)$ for $Z_\bullet(w)\in\R[T]$, 
where $\bullet=\sh$ or $*$, and we often use the notation $Z_\bullet(\bk;T)$ 
for $Z_\bullet(w;T)$ with an index $\bk$ corresponding to $w$.

\smallskip 

Let us introduce  the `star-version' of regularizations. 
The star-harmonic regularization is defined by 
\[\Zbar_*(w;T)=Z_*(w^\star;T)\qquad (w\in\frH^1). \]
On the other hand, the star-shuffle regularization 
$\Zbar_\sh\colon\frH^1\to\R[T]$ is defined differently by 
\[\Zbar_\sh(w;T)=\Zbar_\sh(\bk;T)=Z_\sh\circ W\left(\ \begin{xy}
{(0,0) \ar@{{*}-} (5,0) *+[F]{\bk}}
\end{xy}\ \right),\]
where $\bk$ is the index which corresponds to a word $w$. 

\begin{rem}
If $w\in\frH^0$, we have $\Zbar_*(w)=Z_*(w^\star)=Z(w^\star)$ by definition. 
On the other hand, the identity $\Zbar_\sh(w)=Z_\sh(w^\star)$ 
does \emph{not} hold in general, even if $w$ is in $\frH^0$. 
Our star-shuffle regularization is therefore 
different from the one previously adopted in, for instance, \cite{M}.
\end{rem}

We define the $\R$-linear maps $\rho_Z$ and $\rhobar_Z$ on $\R[T]$ by 
\[\rho_Z(e^{Tu})=A_Z(u)e^{Tu},\quad 
\rhobar_Z(e^{Tu})=A_Z(-u)^{-1}e^{Tu}\]
where 
\[A_Z(u)=\exp\Biggl(\sum_{n=2}^\infty
\frac{(-1)^n}{n}Z(e_n)u^n\Biggr)\in\R[[u]]. \]
In particular, the map $\rho_\zeta$ corresponding to 
the MZV-evaluation map $\zeta$ is exactly the map $\rho$ defined by \eqref{rho}. 

\begin{thm}\label{thm:equiv}
Suppose a $\Q$-linear map $Z\colon\frH^0\to\R$ satisfies the double shuffle relation. 
Then the following three properties of $Z$ are equivalent: 
\begin{alignat}{2}
\label{eq:Reg}
Z_\sh(\bk;T)&=\rho_Z\bigl(Z_*(\bk;T)\bigr)
&&\text{\ for any index $\bk$}, 
\tag{Reg}\\
\label{eq:Reg^star}
\Zbar_\sh(\bk;T)&=\rhobar_Z\bigl(\Zbar_*(\bk;T)\bigr)
&&\text{\ for any index $\bk$}, 
\tag{$\text{Reg}^\star$}\\
\label{eq:FSI}
Z\bigl(\mu\bigl(\bk,\bl\bigr)\bigr)&=Z\bigl(\bk\circledast \bl^\star\bigr)
&&\text{\ for any non-empty indices $\bk$ and $\bl$}. 
\tag{Int-Ser}
\end{alignat}
\end{thm}

We refer to the above three properties as
the \emph{regularization theorem}, the \emph{star-regularization theorem} 
and the \emph{integral-series identity} for $Z$, respectively.  
The proof of Theorem \ref{thm:equiv} is carried out in the next section. 

Because we know the properties \eqref{eq:Reg} and \eqref{eq:FSI} hold
for $Z=\zeta$, we obtain the star-regularization theorem for multiple zeta-star values:

\begin{cor} For any index $\bk$, we have
\[\zeta^\star_\sh(\bk;T)
=\rhobar_\zeta\bigl(\zeta^\star_*(\bk;T)\bigr). \]
\end{cor}

As for the implication  \eqref{eq:FSI} $\Rightarrow$ \eqref{eq:Reg}, it turns out that only `single' shuffle
relation is sufficient, i.e., the following holds.

\begin{thm}\label{thm:shuffle=harmonic}
Suppose that a $\Q$-linear map $Z\colon\frH^0\to\R$ satisfies 
the integral-series identity. Then $Z$ satisfies the shuffle relation $Z(\bk\bsh\bl)=Z(\bk)Z(\bl)$
if and only if $Z$ satisfies the harmonic relation $Z(\bk*\bl)=Z(\bk)Z(\bl)$. 
\end{thm}

The proof is also given in the next section.

\section{Proof of Theorem \ref{thm:equiv} and Theorem \ref{thm:shuffle=harmonic}}\label{sec:proof}
\subsection{Algebraic preliminaries}

To prove Theorem \ref{thm:equiv}, we need some identities 
which are purely algebraic (i.e., hold in $\frH^1$). 

\begin{lem}\label{lem:exp_*}
We have the following identities in the ring of formal power series 
$\frH^1_*[[u]]$:
\begin{align*}
\sum_{n=0}^\infty e_1^nu^n
&=\exp_*\Biggl(\sum_{n=1}^\infty (-1)^{n-1}\frac{e_n}{n}u^n\Biggr)\\
\intertext{and}
\sum_{n=0}^\infty (e_1^n)^\star u^n
&=\exp_*\Biggl(\sum_{n=1}^\infty \frac{e_n}{n}u^n\Biggr),
\end{align*}
where $exp_*$ indicates that the products of coefficients are $*$-products.
\end{lem}

\begin{proof}  
These are more or less well-known. The first (resp.\ second) is essentially an identity 
between elementary (resp.\ complete) and power-sum symmetric functions. See \cite[Theorem 5.1]{H}. 
Also \cite[(4.6)]{IKZ} is equivalent to the first identity.
We may also obtain the second identity from the first by applying the antipode of the Hopf algebra 
of quasi-symmetric functions described in \cite{H2}. 
\end{proof}

For non-empty indices $\bk=(k_1,\ldots,k_r)$ and $\bl=(l_1,\ldots,l_s)$, 
we put 
\begin{gather*}
\bk_i=(k_1,\ldots,k_i),\quad \bk^i=(k_{i+1},\ldots,k_r)\qquad 
(0\leq i\leq r),\\
\rev{\bk}=(k_r,\ldots,k_1),\\
\bk\odot\bl=(k_1,\ldots,k_{r-1},k_r+l_1,l_2,\ldots,l_s). 
\end{gather*}
With our convention, we have $\bk_0=\bk^r=\varnothing$.

\begin{lem}\label{lem:algebraic identities}
For any non-empty indices $\bk=(k_1,\ldots,k_r)$ and $\bl=(l_1,\ldots,l_s)$, 
the following identities hold in $\frH^1$: 
\begin{align}
\label{eq:A_sh}
&\sum_{i=0}^{s-1}(-1)^i\mu(\bk,\bl^i)\bsh \rev{\bl_i}
+(-1)^s\bk\odot\rev{\bl}=0, \tag{$A_\sh$}\\
\label{eq:A_sh^star}
&\sum_{i=0}^{r-1}(-1)^i\mu(\bk^i,\bl)\bsh 
W\left(\ \begin{xy}
{(0,0) \ar@{{*}-} (7,0)*+[F]{\rev{\bk_i}}}
\end{xy}\ \right)
+(-1)^r\,W\left(\ \begin{xy}
{(0,0) \ar@{{*}-} (9,0)*+[F]{\bl\odot\rev{\bk}}}
\end{xy}\ \right)=0, \tag{$A_\sh^\star$}\\
\label{eq:A_*}
&\sum_{i=0}^{s-1}
(-1)^i\bigl(\bk\circledast (\bl^i)^\star\bigr)*\rev{\bl_i}
+(-1)^s\,\bk\odot \rev{\bl}=0, \tag{$A_*$}\\
\label{eq:A_*^star}
&\sum_{i=0}^{r-1}
(-1)^i\bigl(\bk^i\circledast \bl^\star\bigr)*(\rev{\bk_i})^\star
+(-1)^r(\bl\odot\rev{\bk})^\star=0. \tag{$A_*^\star$}
\end{align}
\end{lem}
\begin{proof}
First we show \eqref{eq:A_sh}. By the relation \eqref{eq:shuffle for W}, 
we have 
\[\mu(\bk,\bl^i)\bsh \rev{\bl_i}
=W\left(\ \begin{xy}
{(6,0)*++[o][F]{\bk}="k"},
{(20,6)*++[F]{\bl^i}="l"},
{(0,-6) \ar @{{*}-} "k"}, 
{"k" \ar @{-{o}} (12,6)}, 
{(12,6) \ar @{-} "l"}, 
\end{xy}\ \bigsqcup \ 
\begin{xy}
{(0,-5) \ar@{{*}-} (6,1)*++[o][F]{\rev{\bl_i}}}
\end{xy}\right)=w_{i-1}+w_i, \]
where $w_{-1}=0$ and 
\[w_i
=W\left(\ \begin{xy}
{(6,-6)*++[o][F]{\bk}="k"},
{(20,0)*++[F]{\bl^i}="l"},
{(0,-12) \ar @{{*}-} "k"}, 
{"k" \ar @{-o} (12,0)}, 
{(12,0) \ar @{-} "l"}, 
{"l" \ar @{-} (26,6)}, 
{(26,6) \ar @{{*}-} (32,12) *++[o][F]{\rev{\bl_i}}} 
\end{xy}\right)\]
for $i=0,\ldots,s-1$ (here, the edge connected to the right of 
\fbox{$\bl^i$} is actually connected to the rightmost vertex in 
the diagram \fbox{$\bl^i$}). 
Hence we obtain \eqref{eq:A_sh} by summing them up with signs and 
using that 
$w_{s-1}=W\left(\ \begin{xy}
{(0,-7) \ar@{{*}-} (8,1)*+++[o][F]{\bk\odot\rev{\bl}}}
\end{xy}\right)=\bk\odot\rev{\bl}$. 
The identity \eqref{eq:A_sh^star} can be shown in a similar way. 

Next we prove \eqref{eq:A_*} (again, \eqref{eq:A_*^star} is similarly shown). 
For $i=0,\ldots,s-1$, expand the products $\circledast$ and $*$ in  
\[\bigl(\bk\circledast (\bl^i)^\star\bigr)*\rev{\bl_i}
=\bigl((k_1,\ldots,k_r)\circledast (l_{i+1},\ldots,l_s)^\star\bigr)
*(l_i,\ldots,l_1). \]
We denote by $S_i$ the partial sum of those indices 
in which $l_i$ appears in the right side of $l_{i+1}$, and 
$S'_i$ the sum of the other terms 
(i.e., those in which $l_i$ appears in the left side of $l_{i+1}$ 
or those contain $l_i+l_{i+1}$). When $i=0$, we understand that 
all terms are contained in $S_0$. 
Then we have 
\[S'_0=0,\quad S_i=S'_{i+1} \quad (i=0,\ldots,s-2),\quad 
S_{s-1}=\bk\odot\rev{\bl}. \]
Hence we obtain \eqref{eq:A_*} by taking the alternating sum. 
\end{proof}

\subsection{Proof of Theorem \ref{thm:equiv}}
Now we are ready to prove Theorem \ref{thm:equiv}. 
Here we only show the equivalence of \eqref{eq:Reg^star} and \eqref{eq:FSI}, 
because the equivalence of \eqref{eq:Reg} and \eqref{eq:FSI} is proved  
in almost the same manner. 

First we prove that \eqref{eq:Reg^star} implies \eqref{eq:FSI}.  Throughout,
$Z$ is a $\Q$-linear map from $\frH^0$ to $\R$ satisfying the double shuffle relation.

\begin{prop}
If $Z$ satisfies the star-regularization theorem \eqref{eq:Reg^star}, 
then $Z$ also satisfies the integral-series identity \eqref{eq:FSI}. 
\end{prop}
\begin{proof}
Let $\bk$ and $\bl$ be non-empty indices. 
We prove the equality $Z(\mu(\bk,\bl))=Z(\bk\circledast\bl^\star)$ by 
induction on the depth $r$ of $\bk$. 
Assume the validity for the depth less than $r$ 
(the case $r=1$ is also included, in which case no assumption is needed). 

By applying $Z_\sh$ and $Z_*$ to \eqref{eq:A_sh^star} and 
\eqref{eq:A_*^star} respectively, we obtain 
\begin{align}
\label{eq:Z(A_sh^star)}
\sum_{i=0}^{r-1}
(-1)^iZ\bigl(\mu(\bk^i,\bl)\bigr)\Zbar_\sh(\rev{\bk_i};T)
&+(-1)^r\,\Zbar_\sh(\bl\odot\rev{\bk};T)=0,\\
\label{eq:Z(A_*^star)}
\sum_{i=0}^{r-1}
(-1)^iZ\bigl(\bk^i\circledast\bl^\star\bigr)\Zbar_*(\rev{\bk_i};T)
&+(-1)^r\,\Zbar_*(\bl\odot\rev{\bk};T)=0. 
\end{align}
We then apply $\rhobar_Z$ to \eqref{eq:Z(A_*^star)} to see 
\begin{equation}\label{eq:rhoZ(A_*^star)}
\sum_{i=0}^{r-1}
(-1)^iZ\bigl(\bk^i\circledast\bl^\star\bigr)
\rhobar_Z\bigl(\Zbar_*(\rev{\bk_i};T)\bigr)
+(-1)^r\,\rhobar_Z\bigl(\Zbar_*(\bl\odot\rev{\bk};T)\bigr)=0 
\end{equation}
(note that $Z\bigl(\bk^i\circledast\bl^\star\bigr)\in\R$ and 
$\rhobar_Z$ is $\R$-linear). 
Now compare \eqref{eq:Z(A_sh^star)} and \eqref{eq:rhoZ(A_*^star)}. 
By the assumption \eqref{eq:Reg^star} and the induction hypothesis, 
all terms except for those of $i=0$ in the sums coincide. 
Thus the $i=0$ terms should also coincide, and this is exactly 
what we have to prove. 
\end{proof}

Next we show the converse, i.e., that \eqref{eq:FSI} implies 
\eqref{eq:Reg^star}. In fact, a weaker assumption is sufficient.  

\begin{prop}
If  the equation 
\begin{equation}\label{eq:SI bk=(1,...,1)}
Z\bigl(\mu\bigl((\underbrace{1,\ldots,1}_m),\bl\bigr)\bigr)=Z\bigl((\underbrace{1,\ldots,1}_m)\circledast\bl^\star\bigr)
\end{equation}
holds for any integer $m\geq 1$ and any non-empty index $\bl$, 
then $Z$ satisfies the star-regularization theorem \eqref{eq:Reg^star}. 
\end{prop}
\begin{proof}
First we note that the identities
\begin{equation}\label{eq:Reg^star y^n}
\Zbar_\sh\bigl((\underbrace{1,\ldots,1}_n);T\bigr)
=\frac{T^n}{n!}
=\rhobar_Z\bigl(\Zbar_*\bigl((\underbrace{1,\ldots,1}_n);T\bigr)\bigr)
\end{equation}
hold for any $n\geq 0$, assuming only the double shuffle relation. 
Indeed, the first equality is immediate from the identity 
\[W\left(\ \begin{xy}
{(0,4) \ar @{{*}-} (3,1)}, 
{(3,1) \ar @{{*}.{*}} (8,-4)},
{(1,4) \ar @/^1mm/ @{-}^{n} (8,-3)}
\end{xy}\ \right)
=\frac{\overbrace{e_1\bsh\cdots\bsh e_1}^n}{n!}\]
in $\frH_\sh$.  
On the other hand, by applying $Z_*$ to the second identity of 
Lemma \ref{lem:exp_*}, we obtain 
\[\sum_{n=0}^\infty\Zbar_*\bigl((\underbrace{1,\ldots,1}_n);T\bigr)u^n
=\exp\Biggl(\sum_{n=1}^\infty\frac{Z_*(e_n)}{n}u^n\Biggr)
=e^{Tu}A_Z(-u), \]
hence 
\[\sum_{n=0}^\infty
\rhobar_Z\bigl(\Zbar_*\bigl((\underbrace{1,\ldots,1}_n);T\bigr)\bigr)u^n
=\rhobar_Z(e^{Tu})A_Z(-u)=e^{Tu}. \]
This gives the second equality in \eqref{eq:Reg^star y^n}. 

Next, put $\bk=(\underbrace{1,\ldots,1}_r)$ 
in \eqref{eq:Z(A_sh^star)} and \eqref{eq:Z(A_*^star)} to get 
\begin{align*}
\Zbar_\sh\bigl(\bl\odot(\underbrace{1,\ldots,1}_r);T)
&=\sum_{i=0}^{r-1}(-1)^{r-1-i}
Z\bigl(\mu((\underbrace{1,\ldots,1}_{r-i}),\bl)\bigr)
\Zbar_\sh\bigl((\underbrace{1,\ldots,1}_i);T\bigr),\\
\Zbar_*\bigl(\bl\odot(\underbrace{1,\ldots,1}_r);T)
&=\sum_{i=0}^{r-1}(-1)^{r-1-i}
Z\bigl((\underbrace{1,\ldots,1}_{r-i})\circledast\bl^\star\bigr)
\Zbar_*\bigl((\underbrace{1,\ldots,1}_i);T\bigr)  
\end{align*}
for any non-empty index $\bl$. 
If we apply $\rhobar_Z$ to the latter equality, 
the right-hand side coincides with that of the formar equality, 
by the assumption \eqref{eq:SI bk=(1,...,1)} and the identity 
\eqref{eq:Reg^star y^n}. Hence we have 
\[\Zbar_\sh\bigl(\bl\odot(\underbrace{1,\ldots,1}_r);T)
=\rhobar_Z\bigl(\Zbar_*\bigl(\bl\odot(\underbrace{1,\ldots,1}_r);T)\bigr)\]
for any non-empty index $\bl$ and any integer $r\geq 1$. 
Now the proof is complete, since an arbitrary index is either
of the form $(\underbrace{1,\ldots,1}_n)$ with some $n\geq 0$, 
or of the form $\bl\odot(\underbrace{1,\ldots,1}_r)$ 
with some non-empty $\bl$ and $r\geq 1$. 
\end{proof}

Thus we have finished the proof of Theorem \ref{thm:equiv}. 

\begin{rem}  Our proof of Theorem \ref{thm:equiv} gives an almost purely algebraic
way to prove the fundamental theorem of regularization (Theorem \ref{thm:IKZ}).
The only point where we need the analysis (or rather the property of real numbers)
is in the elementary computation of multiple integrals to prove the integral-series identity \eqref{eq:SI}!
\end{rem}

\subsection{Proof of Theorem \ref{thm:shuffle=harmonic}}
Let us assume that $Z$ satisfies \eqref{eq:FSI} and the shuffle relation, 
and prove the harmonic relation (the other direction is similarly proved). 
We prove that the identity $Z(\bk)Z(\bl)=Z(\bk*\bl)$ holds for any 
admissible indices $\bk$ and $\bl$ by induction on the depth of $\bl$. 

Take any admissible indices $\bk=(k_1,\ldots,k_r)$ and 
$\bl=(l_1,\ldots,l_s)$, with $r,s>0$. We put 
$\tilde{\bk}=(k_1,\ldots,k_{r-1},k_r-1)$ and $\hat{\bl}=(l_s,\ldots,l_1,1)$, 
use \eqref{eq:A_sh} and \eqref{eq:A_*} to the pair $(\tilde{\bk},\hat{\bl})$, 
and apply the map $Z$. Then we obtain that  
\[\sum_{i=0}^s(-1)^iZ\bigl(\mu(\tilde{\bk},\hat{\bl}^i)\bsh \rev{\hat{\bl}_i}\bigr)
=(-1)^sZ\bigl(\tilde{\bk}\odot\rev{\hat{\bl}}\bigr)=
\sum_{i=0}^s(-1)^iZ\bigl((\tilde{\bk}\circledast(\hat{\bl}^i)^\star)*\rev{\hat{\bl}_i}\bigr). \]
Since $\rev{\hat{\bl}_i}=(l_{s-i+1},\ldots,l_s)$ is admissible as well as 
$\mu(\tilde{\bk},\hat{\bl}^i)$ and $\tilde{\bk}\circledast(\hat{\bl}^i)^\star$, 
we can use the assumption of the shuffle relation on the left-hand side 
and the induction hypothesis of the harmonic relation on the right-hand side: 
\[\sum_{i=0}^s(-1)^iZ\bigl(\mu(\tilde{\bk},\hat{\bl}^i)\bigr)\,Z\bigl(\rev{\hat{\bl}_i}\bigr)
=\sum_{i=0}^{s-1}(-1)^iZ\bigl(\tilde{\bk}\circledast(\hat{\bl}^i)^\star\bigr)
\,Z\bigl(\rev{\hat{\bl}_i}\bigr)
+(-1)^sZ\bigl((\tilde{\bk}\circledast(\hat{\bl}^s)^\star)*\rev{\hat{\bl}_s}\bigr). \]
The integral-series identity \eqref{eq:FSI} implies that the corresponding terms 
for $i=0,\ldots,s-1$ are equal, hence we also have the equality for $i=s$: 
\[Z\bigl(\mu(\tilde{\bk},\hat{\bl}^s)\bigr)\,Z\bigl(\rev{\hat{\bl}_s}\bigr)
=Z\bigl((\tilde{\bk}\circledast(\hat{\bl}^s)^\star)*\rev{\hat{\bl}_s}\bigr). \]
This is exactly the identity $Z(\bk)Z(\bl)=Z(\bk*\bl)$, since 
$\mu(\tilde{\bk},\hat{\bl}^s)=\tilde{\bk}\circledast(\hat{\bl}^s)^\star=\bk$ 
and $\rev{\hat{\bl}_s}=\bl$ by definition. 

\section{Relationship with Kawashima's relation}\label{sec:kawashima}
In \cite{K}, Kawashima obtained a remarkable class of 
algebraic relations among MZVs. 
In this section, we show that the double shuffle relation, 
the regularization theorem, and the duality relation together 
imply Kawashima's relation. 

\subsection{The duality relation}
First we formulate the duality relation in the formal setting. 

\begin{defn}
Let us denote by $w\mapsto w^\dual$ the anti-automorphsim 
of $\frH$ determined by $e_0^\dual=e_1$ and $e_1^\dual=e_0$. 
Note that the map $w\mapsto w^\dual$ preserves $\frH^0$. 
We say that a $\Q$-linear map $Z\colon\frH^0\to\R$ 
satisfies the \emph{duality relation} 
if the equality $Z(w)=Z(w^\dual)$ holds for any $w\in\frH^0$. 
\end{defn}

\begin{ex}
The MZV-evaluation map $\zeta\colon\frH^0\to\R$ satisfies the 
duality relation. This is an immediate consequence of the iterated 
integral expression \eqref{eq:SI MZV} and is well known. 
\end{ex}

We need another notion of dual of an index, 
called the Hoffman dual (see \cite{H2}). 

\begin{defn}
For an index $\bk=(k_1,\ldots,k_r)$, its \emph{Hoffman dual} is 
the index $\bk^\vee=(k'_1,\ldots,k'_{r'})$ determined by 
$k:=k_1+\cdots+k_r=k'_1+\cdots+k'_{r'}$ and 
\begin{equation*}
\{1,2,\ldots,k-1\}
=\{k_1,k_1+k_2,\ldots,k_1+\cdots+k_{r-1}\}
\amalg\{k'_1,k'_1+k'_2,\ldots,k'_1+\cdots+k'_{r'-1}\}. 
\end{equation*}
We extend the map $\bk\mapsto\bk^\vee$ to 
a $\Q$-linear automorphism of $\frH^1$. 
\end{defn}

Thus we have two notions of dual $\bk^\dual$ and $\bk^\vee$ 
of an index. They are related to the notion of the transpose of a 2-poset
in the following ways. 

\begin{defn}
For a 2-poset $X$, let $X^\trp$ denote its \emph{transpose}, 
i.e., the 2-poset obtained by reversing the order on $X$ 
and setting $\delta_{X^\trp}(x)=1-\delta_X(x)$. 
We extend the map $X\mapsto X^\trp$ to 
a $\Q$-linear automorphism on $\frP$. 
\end{defn}

The following equalities are easily verified from the definition: 
\begin{alignat*}{2}
W(X)^\dual&=W(X^\trp) &&\text{\ for any $X\in\frP$}, \\
\begin{xy}
{(0,0) \ar @{{*}-} (7,0)*+[F]{\bk^\vee}}
\end{xy}&=
\left(\ \begin{xy}
{(0,0) \ar @{o-} (6,0)*+[F]{\bk}}
\end{xy}\ \right)^\trp
&&\text{\ for any non-empty index $\bk$}. 
\end{alignat*}
From the first equality, we see that 
$Z$ satisfies the duality relation if and only if 
\[Z(W(X))=Z(W(X^\trp))\]
holds for any $X\in\frP^0$. 

\subsection{Kawashima's relation}
Let us recall Kawashima's result in our notation. 

\begin{thm}[{\cite[Theorem 5.3]{K}}]
For any non-empty indices $\bk$, $\bl$ and any integer $m\geq 1$, we have 
\[\begin{split}
\sum_{\substack{p,q\geq 1\\ p+q=m}}
\zeta\bigl((\underbrace{1,\ldots,1}_p)
\circledast (\bk^\vee)^\star\bigr)
&\zeta\bigl((\underbrace{1,\ldots,1}_q)
\circledast (\bl^\vee)^\star\bigr)\\
&=-\zeta\bigl((\underbrace{1,\ldots,1}_m)
\circledast ((\bk\bast\bl)^\vee)^\star\bigr). 
\end{split}\]
\end{thm}
Here, the multiplication $\bast$ on $\frH^1$ is the ``zeta-star version'' of 
the harmonic product and is defined by
\begin{gather*}
1\bast w=w\bast 1=w\quad (w\in\frH^1), \\
e_kv\bast e_lw=e_k(v\bast e_lw)+e_l(e_kv\bast w)-e_{k+l}(v\bast w) 
\quad (k,l\geq 1, v,w\in\frH^1). 
\end{gather*}

Motivated by the above result, 
we give the following definition: 

\begin{defn}
A $\Q$-linear map $Z\colon\frH^0\to\R$ is said to satisfy 
\emph{Kawashima's relation} if 
\[\begin{split}
\sum_{\substack{p,q\geq 1\\ p+q=m}}
Z\bigl((\underbrace{1,\ldots,1}_p)
\circledast (\bk^\vee)^\star\bigr)
&Z\bigl((\underbrace{1,\ldots,1}_q)
\circledast (\bl^\vee)^\star\bigr)\\
&=-Z\bigl((\underbrace{1,\ldots,1}_m)
\circledast ((\bk\bast\bl)^\vee)^\star\bigr) 
\end{split}\]
holds for any non-empty indices $\bk,\bl$ and any integer $m\geq 1$. 
\end{defn}

\begin{thm}
If a $\Q$-linear map $Z\colon\frH^0\to\R$ satisfies 
the double shuffle relation, the regularization theorem 
(or equivalently the integral-series identity) and the duality relation, 
then $Z$ also satisfies Kawashima's relation. 
\end{thm}
\begin{proof}
By using \eqref{eq:FSI} and the duality relation, we have 
\begin{align*}
Z\bigl((\underbrace{1,\ldots,1}_m)
\circledast (\bk^\vee)^\star\bigr)
=Z\circ W\left( \begin{xy}
{(0,-5) \ar @{{*}.} (4,-1)}, 
{(4,-1) \ar @{{*}-} (8,3)}, 
{(8,3) \ar @{o-} (15,3)*+[F]{\bk^\vee}}, 
{(0,-4) \ar @/^1mm/ @{-}^{m} (3,-1)} 
\end{xy}\ \right)
=Z\circ W\left( \begin{xy}
{(0,5) \ar @{{o}.} (4,1)}, 
{(4,1) \ar @{{o}-} (8,-3)}, 
{(8,-3) \ar @{{*}-} (14,-3)*+[F]{\bk}}, 
{(0,4) \ar @/_1mm/ @{-}_{m} (3,1)} 
\end{xy}\ \right) 
\end{align*}
for non-empty $\bk$ and $m\geq 1$. 
Therefore, if we denote the rightmost expression by $A_m(\bk)$, 
it suffices to prove 
\begin{equation}\label{eq:A_m relation}
\sum_{\substack{p,q\geq 1\\ p+q=m}}
A_p(\bk)A_q(\bl)=-A_m(\bk\bast\bl) 
\end{equation}
(here $A_m$ is extended to a linear functional on $e_1\frH^1$; 
the same rule is also applied to $B_m$ below). 
To prove this, we need two lemmas. 

\begin{lem}\label{lem:B_m relation}
For an index $\bk$ (admissible or not) and $m\geq 0$, put 
\[B_m(\bk)=Z_*\circ W\left( \begin{xy}
{(0,4) \ar @{{o}.} (4,0)}, 
{(4,0) \ar @{{o}-} (8,-4)}, 
{(8,-4) \ar @{{*}-} (13,1)*++[o][F]{\bk}}, 
{(0,3) \ar @/_1mm/ @{-}_{m} (3,0)}, 
\end{xy}\ \right) \]
if $\bk\ne\varnothing$, and put 
\[B_m(\varnothing)=\begin{cases}1&(m=0), \\ 0&(m>0).\end{cases}\]
Then we have 
\begin{equation}\label{eq:B_m relation}
\sum_{\substack{p,q\geq 0\\ p+q=m}}
B_p(\bk)B_q(\bl)=B_m(\bk*\bl) 
\end{equation}
for any indices $\bk,\bl$ and $m\geq 0$. 
\end{lem}
\begin{proof}
If $\bk$ or $\bl$ is empty, the claim is obvious. 
If both $\bk=(k_1,\ldots,k_r)$ and $\bl=(l_1,\ldots,l_s)$ are non-empty, 
it is obtained by applying $Z_*$ to the identity 
\begin{equation}\label{eq:B_m relation word}
\sum_{\substack{p,q\geq 0\\ p+q=m}}
W\left( \begin{xy}
{(0,4) \ar @{{o}.} (4,0)}, 
{(4,0) \ar @{{o}-} (8,-4)}, 
{(8,-4) \ar @{{*}-} (13,1)*++[o][F]{\bk}}, 
{(0,3) \ar @/_1mm/ @{-}_{p} (3,0)}, 
\end{xy}\ \right)
*W\left( \begin{xy}
{(0,4) \ar @{{o}.} (4,0)}, 
{(4,0) \ar @{{o}-} (8,-4)}, 
{(8,-4) \ar @{{*}-} (13,1)*++[o][F]{\bl}}, 
{(0,3) \ar @/_1mm/ @{-}_{q} (3,0)}, 
\end{xy}\ \right)
=W\left( \begin{xy}
{(0,4) \ar @{{o}.} (4,0)}, 
{(4,0) \ar @{{o}-} (8,-4)}, 
{(8,-4) \ar @{{*}-} (13,1)*++[o][F]{\bk*\bl}}, 
{(0,3) \ar @/_1mm/ @{-}_{m} (3,0)}, 
\end{xy}\ \right)
\end{equation}
in $\frH^1$. To see \eqref{eq:B_m relation word}, we expand $W$'s and 
the harmonic products. 
First consider the factor $W\left( \begin{xy}
{(0,4) \ar @{{o}.} (4,0)}, 
{(4,0) \ar @{{o}-} (8,-4)}, 
{(8,-4) \ar @{{*}-} (13,1)*++[o][F]{\bk}}, 
{(0,3) \ar @/_1mm/ @{-}_{p} (3,0)}, 
\end{xy}\ \right)$. This is the finite sum of words in $e_0$ and $e_1$ 
which are obtained by inserting $e_0$ into (or putting at the right of) 
\[e_1\underbrace{e_0\cdots e_0}_{k_1-1}e_1\underbrace{e_0\cdots e_0}_{k_2-1}
\cdots e_1\underbrace{e_0\cdots e_0}_{k_r-1} \]
$p$ times. Take such a word $v$ with a specified way of insertion, 
and similarly $w$ from the second factor $W\left( \begin{xy}
{(0,4) \ar @{{o}.} (4,0)}, 
{(4,0) \ar @{{o}-} (8,-4)}, 
{(8,-4) \ar @{{*}-} (13,1)*++[o][F]{\bl}}, 
{(0,3) \ar @/_1mm/ @{-}_{q} (3,0)}, 
\end{xy}\ \right)$. 
Then their harmonic product $v*w$ is the sum of words 
each of which is obtained from 
\[e_1\underbrace{e_0\cdots e_0}_{h_1-1}e_1\underbrace{e_0\cdots e_0}_{h_2-1}
\cdots e_1\underbrace{e_0\cdots e_0}_{h_t-1}, \]
where $(h_1,\ldots,h_t)$ is an index appearing in $\bk*\bl$, 
by inserting $e_0$ $p+q=m$ times in the following way. 
For $a=1,\ldots,t$, $h_a=k_i$, $l_j$ or $k_i+l_j$ for some $i,j$. 
In the first two cases, insert $e_0$ into 
\[e_1\underbrace{e_0\cdots e_0}_{h_a-1}
=e_1\underbrace{e_0\cdots e_0}_{k_i-1} \ \text{ or } \ 
e_1\underbrace{e_0\cdots e_0}_{l_j-1}\] in the manner specified for 
the corresponding part of $v$ or $w$ respectively. 
In the third case, insert $e_0$ into 
\[e_1\underbrace{e_0\cdots e_0}_{h_a-1}
=e_1\underbrace{e_0\cdots e_0}_{k_i-1}\ 
e_0\underbrace{e_0\cdots e_0}_{l_j-1}\]
in the manner specified for the corresponding parts of $v$ and $w$. 

Thus we associate with each word appearing in the expansion of 
the left-hand side of \eqref{eq:B_m relation word} 
a word appearing in the right-hand side, 
and it is easy to see that this correspondence is bijective. 
Hence we have the equality \eqref{eq:B_m relation word}. 
\end{proof}

\begin{lem}\label{lem:A_mB_0}
For a non-empty index $\bk=(k_1,\ldots,k_r)$ 
and an integer $m\geq 1$, we have 
\begin{equation}\label{eq:A_mB_0}
\sum_{i=0}^{r-1}(-1)^iA_m(\bk^i)B_0(\rev{\bk_i})+(-1)^rB_m(\rev{\bk})=0. 
\end{equation}
\end{lem}
\begin{proof}
It can be shown that 
\begin{align*}
\sum_{i=0}^{r-1}(-1)^i
W\left( \begin{xy}
{(0,6) \ar @{{o}.} (4,2)}, 
{(4,2) \ar @{{o}-} (8,-2)}, 
{(8,-2) \ar @{{*}-} (14,-2)*++[F]{\bk^i}}, 
{(0,5) \ar @/_1mm/ @{-}_{m} (3,2)}, 
\end{xy}\ \right) 
\bsh W\left(\ \begin{xy}
{(8,-4) \ar @{{*}-} (14,2)*++[o][F]{\rev{\bk_i}}}, 
\end{xy}\right)
+(-1)^r W\left( \begin{xy}
{(0,4) \ar @{{o}.} (4,0)}, 
{(4,0) \ar @{{o}-} (8,-4)}, 
{(8,-4) \ar @{{*}-} (14,2)*++[o][F]{\rev{\bk}}}, 
{(0,3) \ar @/_1mm/ @{-}_{m} (3,0)}, 
\end{xy}\ \right)=0
\end{align*}
by a method similar to the proof of \eqref{eq:A_sh} 
in Lemma \ref{lem:algebraic identities}. 
Then, applying $Z_\sh$, we have 
\begin{align*}
\sum_{i=0}^{r-1}(-1)^i
A_m(\bk^i)\, Z_\sh\circ W\left(\ \begin{xy}
{(8,-4) \ar @{{*}-} (14,2)*++[o][F]{\rev{\bk_i}}}, 
\end{xy}\right)
+(-1)^rZ_\sh\circ W\left( \begin{xy}
{(0,4) \ar @{{o}.} (4,0)}, 
{(4,0) \ar @{{o}-} (8,-4)}, 
{(8,-4) \ar @{{*}-} (14,2)*++[o][F]{\rev{\bk}}}, 
{(0,3) \ar @/_1mm/ @{-}_{m} (3,0)} 
\end{xy}\ \right)=0. 
\end{align*}
Finally, we apply $\rho_Z^{-1}$ and use the assumption 
$Z_*=\rho_Z^{-1}\circ Z_\sh$ (i.e.\ the regularization theorem for $Z$) 
to obtain \eqref{eq:A_mB_0}. 
\end{proof}

Define an $\R$-linear operator $R$ on $\frH^1$ by 
$R(\bk)=(-1)^r\rev{\bk}$ for indices $\bk$ of depth $r$. 
It is immediate from the definition that 
\begin{equation}\label{eq:R(bast)}
R(\bk\bast\bl)=R(\bk)*R(\bl)
\end{equation}
holds for any indices $\bk$ and $\bl$. 
We set $B_m'(\bk)=B_m(R(\bk))$. 

Now let us return to the proof of \eqref{eq:A_m relation}. 
For non-empty indices $\bk=(k_1,\ldots,k_r)$ and 
$\bl=(l_1,\ldots,l_s)$, apply Lemma \ref{lem:A_mB_0} to 
(each term of) $\bk\bast\bl$ to see 
\begin{equation*}
\begin{split}
\sum_{i=0}^{r-1}\sum_{j=0}^{s-1}
A_m(\bk^i\bast\bl^j)&B_0'(\bk_i\bast\bl_j)
+\sum_{i=0}^{r-1}A_m(\bk^i)B_0'(\bk_i\bast\bl)\\
&+\sum_{j=0}^{s-1}A_m(\bl^j)B_0'(\bk\bast\bl_j)
+B_m'(\bk\bast\bl)=0. 
\end{split}
\end{equation*}
By \eqref{eq:R(bast)} and Lemma \ref{lem:B_m relation}, 
this can be rewritten as 
\begin{equation}\label{eq:A_m rel proof}
\begin{split}
\sum_{i=0}^{r-1}\sum_{j=0}^{s-1}
A_m(\bk^i\bast\bl^j)&B_0'(\bk_i)B_0'(\bl_j)
+\sum_{i=0}^{r-1}A_m(\bk^i)B_0'(\bk_i)B_0'(\bl)\\
&+\sum_{j=0}^{s-1}A_m(\bl^j)B_0'(\bk)B_0'(\bl_j)
+\sum_{\substack{p,q\geq 0\\ p+q=m}}B_p'(\bk)B_q'(\bl)=0. 
\end{split}
\end{equation}
By Lemma \ref{lem:A_mB_0}, the second and the third sums are cancelled out 
with the terms for $(p,q)=(m,0)$ and $(0,m)$ in the fourth sum: 
\begin{align*}
&\sum_{i=0}^{r-1}A_m(\bk^i)B_0'(\bk_i)B_0'(\bl)
+B_m'(\bk)B_0'(\bl)=0, \\
&\sum_{j=0}^{s-1}A_m(\bl^j)B_0'(\bk)B_0'(\bl_j)
+B_0'(\bk)B_m'(\bl)=0. 
\end{align*}
Furthermore, we also apply Lemma \ref{lem:A_mB_0} 
to the other terms in the fourth sum to obtain 
\[\sum_{\substack{p,q\geq 1\\ p+q=m}}
B_p'(\bk)B_q'(\bl)
=\sum_{\substack{p,q\geq 1\\ p+q=m}}
\sum_{i=0}^{r-1}\sum_{j=0}^{s-1}
A_p(\bk^i)A_q(\bl^j)B_0'(\bk_i)B_0'(\bl_j). \]
Therefore, \eqref{eq:A_m rel proof} becomes 
\[\sum_{i=0}^{r-1}\sum_{j=0}^{s-1}
\Biggl\{A_m(\bk^i\bast\bl^j)
+\sum_{\substack{p,q\geq 1\\ p+q=m}}A_p(\bk^i)A_q(\bl^j)\Biggr\}
B_0'(\bk_i)B_0'(\bl_j)=0. \]
Now we can prove \eqref{eq:A_m relation} by induction on $r$ and $s$. 
In fact, if $i>0$, the depth of $\bk^i$ is less than $r$, 
hence the induction hypothesis implies that 
the corresponding value in $\{\ \ \}$ vanishes. 
The same holds for terms with $j>0$. 
Hence also the term of $i=j=0$ must vanish, 
and this is exactly what we have to show. 
\end{proof}

\section{A proof of the restricted sum formula using the Int-Ser identity}
\label{sec:sumformula}

The sum formula for the multiple zeta values, which asserts that the sum 
of all MZVs with fixed weight and depth is equal to the Riemann zeta value 
$\zeta(k)$ ($k=$ weight), is probably the most well-known identity in the 
theory, and a good many different proofs are known. 
In the light of our Conjecture \ref{ISconj}, there should be yet another proof 
based on the integral-series identity \eqref{eq:SI}. 
In this last section, we deduce from the identity \eqref{eq:SI} 
a generalization of the sum formula provided by Eie, Liaw and Ong \cite{ELO}, 
which they call the restricted sum formula. 

Let $Z\colon\frH^0\to\R$ be a $\Q$-linear map 
(we \emph{don't} assume the double shuffle relation here). 

\begin{prop} 
If $Z$ satisfies the integral-series identity \eqref{eq:FSI}, 
then the restricted sum formula 
\begin{equation}\label{eq:RSF}
\begin{split}
&\sum_{\substack{k_1,\ldots,k_q\geq 1\\ 
k_1+\cdots+k_q=k-p}}
Z(\underbrace{1,\ldots,1}_{p-1},k_1,\ldots,k_{q-1},k_q+1)\\
&=\sum_{\substack{k_1,\ldots,k_p\geq 1\\ 
k_1+\cdots+k_p=p+q-1}}
Z(k_1,\ldots,k_{p-1},k_p+k-p-q+1)
\end{split}
\end{equation}
holds for any positive integers $k,p,q$ with $k\geq p+q$. 
\end{prop}

\begin{proof} 
Write $S(k,p,q)$ (resp.\ $T(k,p,q)$) 
for the left-hand side (resp.\ the right-hand side) of \eqref{eq:RSF}, 
and put $m=k-p-q+1$. 
Take $\bk=(\underbrace{1,\ldots,1}_{p-1},m)$ and 
$\bl=(\underbrace{1,\ldots,1}_q)$. 
Then we see that 
\begin{align}
\notag 
Z(\mu(\bk,\bl))&=Z\circ W\left(\ 
\begin{xy}
{(0,-10) \ar@{{*}-} (2,-7.5)}, 
{(2,-7.5) \ar@{{*}.} (6,-2.5)}, 
{(6,-2.5) \ar@{{*}-} (8,0)}, 
{(8,0) \ar@{o.} (12,5)}, 
{(12,5) \ar@{o-} (14,7.5)}, 
{(14,7.5) \ar@{o-} (16,5)}, 
{(16,5) \ar@{{*}.} (20,0)}, 
{(20,0) \ar@{{*}-{*}} (22,-2.5)}, 
{(-0.5,-9) \ar @/^1.5mm/ @{-}^{p} (5,-2)}, 
{(7.5,1) \ar @/^1.5mm/ @{-}^{m} (13,8)}, 
{(17,5.5) \ar @/^1mm/ @{-}^{q-1} (22.5,-1.5)}
\end{xy}\right)
=\sum_{i=0}^{q-1}Z\circ W\left( 
\begin{xy}
{(0,-10) \ar@{{*}.} (4,-5)}, 
{(4,-5) \ar@{{*}-} (6,-2.5)}, 
{(6,-2.5) \ar@{{*}-} (2,0)}, 
{(2,0) \ar@{o.} (2,6.5)}, 
{(2,6.5) \ar@{o-} (6,9)}, 
{(6,9) \ar@{o-} (10,6.5)}, 
{(10,6.5) \ar@{{*}.} (10,0)}, 
{(10,0) \ar@{{*}-} (6,-2.5)}, 
{(6,-2.5) \ar@{-} (8,-5)}, 
{(8,-5) \ar@{{*}.{*}} (12,-10)}, 
{(-0.5,-9) \ar @/^1mm/ @{-}^{p-1} (3,-4.5)}, 
{(1,0) \ar @/^1mm/ @{-}^{m-1} (1,6.5)}, 
{(11,6.5) \ar @/^1mm/ @{-}^{q-1-i} (11,0)}, 
{(9,-4.5) \ar @/^1mm/ @{-}^{i} (12.5,-9)}, 
\end{xy}\right)\\
\label{muvalue}
&=\sum_{i=0}^{q-1}\binom{p+i-1}{p-1}S(k,p+i,q-i) 
\end{align}
holds.  On the other hand, we have 
\begin{equation}\label{csvalue}
Z(\bk\circledast\bl^\star)
=Z\bigl((\underbrace{1,\ldots,1}_{p-1},m)
\circledast(\underbrace{1,\ldots,1}_q)^\star\bigr)
=\sum_{j=p}^{p+q-1}\binom{j-1}{p-1}T(k,j,p+q-j). 
\end{equation}
To see this, we first note that $(\underbrace{1,\ldots,1}_q)^\star$ is 
equal to the formal sum of all indices of weight $q$. 
An index appearing in the expansion of $\bk\circledast\bl^\star$ 
has weight $k$, depth $j$ with $p\le j\le p+q-1$, 
and the last component greater than $m$. 
Given such an index $\mathbf{m}$, we may count the number of appearances 
of $\mathbf{m}$ in the expansion of $\bk\circledast\bl^\star$ 
as $\binom{j-1}{p-1}$ because there are $\binom{j-1}{p-1}$ choices of 
the position of 1 in  $\mathbf{m}$ coming from $\bk$, 
and to each such choice there is a unique index in (the expansion of) $\bl^\star$ 
to be combined by $\circledast$ to make $\mathbf{m}$. 
Therefore, by \eqref{eq:FSI}, we obtain from 
\eqref{muvalue} and \eqref{csvalue} 
\begin{align*}
\sum_{i=0}^{q-1}\binom{p+i-1}{p-1}S(k,p+i,q-i)
&=\sum_{j=p}^{p+q-1}\binom{j-1}{p-1}T(k,j,p+q-j)\\
&\underset{j=p+i}{=}\sum_{i=0}^{q-1}\binom{p+i-1}{p-1}T(k,p+i,q-i). 
\end{align*}
Hence, by induction on $q$, the proposition follows. 
\end{proof}

\begin{cor}
The integral-series identity \eqref{eq:FSI} implies the sum formula 
\[\sum_{\substack{k_1,\ldots,k_{q-1}\geq 1,k_q\geq 2\\ 
k_1+\cdots+k_q=k}}Z(k_1,\ldots,k_q)=Z(k)\]
for any integers $k>q>0$. 
\end{cor}
\begin{proof}
Set $p=1$ in \eqref{eq:RSF}. 
\end{proof}

\begin{rem} We can also deduce Hoffman's relation (\cite{H0})
$$ \sum_{i=1}^r \zeta(k_1,\ldots,k_i+1,\ldots,k_r)=\sum_{1\le i\le r \atop 
k_i\ge2} \sum_{j=1}^{k_i-1}\zeta(k_1,\ldots,k_{i-1},j,k_i-j+1,k_{i+1},\ldots,k_r)$$
more straightforwardly from the integral-series identity  \eqref{eq:FSI}  by taking 
$\bk=(k_1,\ldots,k_{r-1},k_r-1)$ and $\bl=(1,1)$.
(This was also communicated to us by Henrik Bachman.)
To deduce known identities such as Ohno's relation \cite{Ohno}
directly from  \eqref{eq:FSI} may be a good challenge for
graduate students.  

\end{rem}
\section*{Acknowledgements}
This work was supported in part by JSPS KAKENHI Grant Numbers 
JP24224001, JP23340010, JP26247004, JP16H06336, 
as well as JSPS Joint Research Project with CNRS ``Zeta functions of 
several variables and applications,'' 
JSPS Core-to-Core program ``Foundation of a Global Research Cooperative 
Center in Mathematics focused on Number Theory and Geometry" and 
the KiPAS program 2013--2018 of the Faculty of Science and Technology 
at Keio University.

\noindent
Faculty of Mathematics, Kyushu University \\
744 Motooka, Nishi-ku, Fukuoka, 819-0395, JAPAN\\
e-mail: \texttt{mkaneko@math.kyushu-u.ac.jp}

\

\noindent
Keio Institute of Pure and Applied Sciences (KiPAS), \\
Graduate School of Science and Technology, Keio University\\
3-14-1 Hiyoshi, Kohoku-ku, Yokohama, 223-8522, JAPAN\\
e-mail: \texttt{yamashu@math.keio.ac.jp}

\end{document}